\documentclass[12pt,reqno]{amsart}
\allowdisplaybreaks

\usepackage{amsmath,amstext,amsthm,amssymb,amsxtra,color}
\usepackage[top=1.5in, bottom=1.5in, left=1.25in, right=1.25in]	{geometry}
\usepackage{txfonts} 
\usepackage[T1]{fontenc}
\usepackage{lmodern}

 \usepackage{euler}   
\usepackage{pdfsync}
\usepackage{multicol}

\usepackage{graphicx}
\graphicspath{ {images/} }

\usepackage[backrefs]{amsrefs}

\begin{document}


\newcommand{\norm}[1]{\ensuremath{\left\|#1\right\|}}
\newcommand{\abs}[1]{\ensuremath{\left\vert#1\right\vert}}
\newcommand{\ip}[2]{\ensuremath{\left\langle#1,#2\right\rangle}}
\newcommand{\p}{\ensuremath{\partial}}
\newcommand{\pr}{\mathcal{P}}

\newcommand{\pbar}{\ensuremath{\bar{\partial}}}
\newcommand{\db}{\overline\partial}
\newcommand{\D}{\mathbb{D}}
\newcommand{\B}{\mathbb{B}}
\newcommand{\Sp}{\mathbb{S}}
\newcommand{\T}{\mathbb{T}}
\newcommand{\R}{\mathbb{R}}
\newcommand{\Z}{\mathbb{Z}}
\newcommand{\C}{\mathbb{C}}
\newcommand{\N}{\mathbb{N}}
\newcommand{\mQ}{\mathcal{Q}}
\newcommand{\mS}{\mathcal{S}}
\newcommand{\scrH}{\mathcal{H}}
\newcommand{\scrL}{\mathcal{L}}
\newcommand{\td}{\widetilde\Delta}
\newcommand{\pw}{\text{PW}}
\newcommand{\esup}{\text{ess.sup}}
\newcommand{\Tn}{\mathcal{T}_n}
\newcommand{\Bn}{\mathbb{B}_n}
\newcommand{\rt}{\mathcal{O}}
\newcommand{\avg}[1]{\langle #1 \rangle}
\newcommand{\one}{\mathbbm{1}}
\newcommand{\eps}{\varepsilon}

\newcommand{\La}{\langle }
\newcommand{\Ra}{\rangle }
\newcommand{\rk}{\operatorname{rk}}
\newcommand{\card}{\operatorname{card}}
\newcommand{\ran}{\operatorname{Ran}}
\newcommand{\osc}{\operatorname{OSC}}
\newcommand{\im}{\operatorname{Im}}
\newcommand{\re}{\operatorname{Re}}
\newcommand{\tr}{\operatorname{tr}}
\newcommand{\vf}{\varphi}
\newcommand{\f}[2]{\ensuremath{\frac{#1}{#2}}}

\newcommand{\kzp}{k_z^{(p,\alpha)}}
\newcommand{\klp}{k_{\lambda_i}^{(p,\alpha)}}
\newcommand{\TTp}{\mathcal{T}_p}
\newcommand{\m}[1]{\mathcal{#1}}
\newcommand{\md}{\mathcal{D}}
\newcommand{\qan}{\abs{Q}^{\alpha/n}}
\newcommand{\sbump}[2]{[[ #1,#2 ]]}
\newcommand{\mbump}[2]{\lceil #1,#2 \rceil}
\newcommand{\cbump}[2]{\lfloor #1,#2 \rfloor}
\newcommand{\unit}{1\!\!1}


\newcommand{\entrylabel}[1]{\mbox{#1}\hfill}

\newenvironment{entry}
{\begin{list}{X}%
  {\renewcommand{\makelabel}{\entrylabel}%
      \setlength{\labelwidth}{55pt}%
      \setlength{\leftmargin}{\labelwidth}
      \addtolength{\leftmargin}{\labelsep}%
   }%
}%
{\end{list}}


\numberwithin{equation}{section}

\newtheorem{thm}{Theorem}[section]
\newtheorem{lm}[thm]{Lemma}
\newtheorem{cor}[thm]{Corollary}
\newtheorem{conj}[thm]{Conjecture}
\newtheorem{prob}[thm]{Problem}
\newtheorem{prop}[thm]{Proposition}
\newtheorem*{prop*}{Proposition}
\newtheorem{claim}[thm]{Claim}

\theoremstyle{remark}
\newtheorem{rem}[thm]{Remark}
\newtheorem*{rem*}{Remark}
\newtheorem{defn}[thm]{Definition}

\hyphenation{geo-me-tric}

\title[Weyl Asymptotics]{Weyl Asymptotics for Perturbations of Morse Potential and 
Connections to the Riemann Zeta Function}

\author[]{Rob Rahm}
\address{Texas A\&M University}
\email{robertrahm@gmail.com}



\begin{abstract}
Let $N(T;V)$ denote the number of eigenvalues of the Schr\"odinger operator $-y'' + Vy$ with absolute value 
less than $T$. This paper studies the Weyl asymptotics of perturbations of the Schr\"odinger operator 
$-y'' + \frac{1}{4}e^{2t}y$ on $[x_0,\infty)$. In particular, we show that perturbations by functions $\varepsilon(t)$ that 
satisfy $\abs{\varepsilon(t)}\lesssim e^{t}$ do not change the Weyl asymptotics very much. Special emphasis is placed on connections to the asymptotics of the 
zeros of the Riemann zeta function. 
\end{abstract}

\maketitle

\section{Introduction}
It is known that if the Riemann hypothesis is true, then there is a positive semi--definite 
matrix $H(x)=\begin{pmatrix}h_{1}(x)&h_2(x)\\h_2(x)&h_3(x)\end{pmatrix}$ such that the spectrum of:
\begin{align*}
\begin{pmatrix}0&-1\\1&0\end{pmatrix}\begin{pmatrix}y_1'(x)\\y_2'(x)\end{pmatrix}
=zH(x)\begin{pmatrix}y_1(x)\\y_2(x)\end{pmatrix}
\end{align*}
(with self--adjoint boundary conditions on an interval $[a,b]$; here $b\leq\infty$), is the same as the imaginary parts of the 
zeros of the Riemann $\zeta$ function on the line $\frac{1}{2}+it$. If the coefficients are smooth enough, then 
this can be transformed to a Schr\"odinger operator on $[a,b]$ that squares the eigenvalues (see for example 
\cite{Lag2006,Lag2009}). 

The purpose of this paper is to indicate some properties that this potential -- if it exists -- must have. It 
is well--known that if $Z(T)$ represents the number of zeros of magnitude less than $\abs{T}$ of the $\zeta$ function, 
then:
\begin{align*}
Z(T) = \frac{1}{\pi}T\log{T} + \frac{1}{\pi}(-2\log{2\pi} - 1)T + O(\log{T}).
\end{align*}
If a potential, $V$, exists whose eigenvalues are the squares of the imaginary parts of the zeros of $\zeta$, and if 
$N(T,L_V)$ represents the number of eigenvalues less than $T$, then the Weyl--asymptotics would satisfy:
\begin{align*}
Z(T) = \frac{1}{\pi}\sqrt{T}\log{\sqrt{T}} + \frac{1}{\pi}(-2\log{2\pi} - 1)\sqrt{T} + O(\log{\sqrt{T}})
\end{align*}
We use this to give some properties that such a potential must have to match the asymptotics in this way. 

Much of the work in this paper builds off of work by Lagarias in papers such as \cites{Lag2006,Lag2009}. 
In \cite{Lag2009}, it is shown that if $V(t)=\frac{1}{4}e^{2t} + ke^{t}$ and if $L_V:=L_{V,x_0,\alpha}f = -f'' +Vf$ is 
the associated Schr\"odinger operator on the interval $[x_0,\infty)$, and if $N(T;V,x_0,\alpha)$ is the number of 
eigenvalues of $L_V$ less than $\abs{T}$, then
\begin{thm}\label{T:lag}
For $L_{Q_{k},x_0,\alpha}$ the Weyl Asymptotics satisfy:
\begin{align}\label{E:lag_asym}
N(T;V,x_0) = \frac{1}{\pi}\sqrt{T}\log{\sqrt{T}} + \frac{1}{\pi}(2\log2 - 1 - x_0)\sqrt{T} + O(1), 
\end{align}
as $T\to\infty$. The constant in the $O(1)$ depends on $k$ and $x_0$. 
\end{thm}

There is an $O(1)$ error term here, whereas in the Weyl--asymptotics for a potential that encodes the zeros 
of $\zeta$ should include an $O(\log\sqrt{T})$ term. We show in this paper that if $V$ is ''close'' to the 
potential $\frac{1}{4}e^{2t}+ke^{t}$, then there is no $O(\log\sqrt{T})$ in the Weyl--asymptotics. On the other hand, 
we show that if $V$ is ``too far'' from $V$, then there is an error term that is bigger than $O(\log\sqrt{T})$. These
vague categories of ``too close'' and ``too far'' (which are defined below) do not form a dichotomy, and so a 
potential -- if it exists -- that matches the asymptotics appropriately must be neither ''too close'' or ''too far''.

Our first theorem shows that any perturbation by a function 
$\eps(t)$ that satisfies $\abs{\eps(t)}\lesssim e^{t}$ will not really change the asymptotics (below, 
$Q_0(t)=\frac{1}{4}e^{2t}$):
\begin{thm}\label{T:converse}
Let $\eps(t)$ be a function that satisfies $\abs{\eps(t)}\lesssim e^{t}$. Then there holds: 
\begin{align*}
N(T; Q_0 + \eps, x_0, \alpha) 
= \frac{1}{\pi}\sqrt{T}\log{\sqrt{T}} + \frac{1}{\pi}(2\log2 - 1 - x_0)\sqrt{T} + O(1),
\end{align*}
where the constant in the $O(1)$ depends on $x_0$ and $\eps$. In particular:
\begin{align*}
\abs{N(T; Q_0,x_0, \alpha) - N(T; Q_0 + \eps, x_0, \alpha)} < O(1), 
\end{align*}
and so no $\log{\sqrt{T}}$ is introduced. 
\end{thm}

On the other hand, we prove that perturbations 
of $Q_0$ by functions that are bigger (resp. smaller) than $e^{(1+\delta)x}$ (resp. $-e^{(1+\delta)x}$) for some 
$\delta>0$ introduce a term that is on the order of (at least) $T^{\frac{\eps}{4}}\sqrt{\log T}$:
\begin{thm}\label{T:mt}
Let $V$ be a function that satisfies $V(t) \geq \frac{1}{4}e^{2t} + \frac{1}{4}e^{(1+\eps)t}$ 
(or $V(t) \leq \frac{1}{4}e^{2t} + \frac{1}{4}e^{(1+\eps)t}$) for some $\eps>0$, then we have:
\begin{align*} 
\abs{N(T; V, x_0, \alpha) - N(T; Q_{0}, x_0, \alpha)} \gtrsim T^{\frac{\eps}{4}}\sqrt{\log T}.
\end{align*}
\end{thm}

Finally, we prove that if $V(t)$ is sub-exponential on a fixed percentage of every interval $[x_0, R]$, then 
$N(T;V,x_0)$ does not even match the assymptotics in the first order:


\begin{thm}\label{T:bigoh}
Let $V$ be a real potential and suppose there is a sub--exponential function $W(t)$ (by 
sub--exponential we mean $\frac{\log W(t)}{t}\to 0$ as $t\to\infty$) such 
that that there is a positive number, $\delta$ such that for all 
sufficiently large $R$ there holds $\abs{\{t: V(t) < W(t)\}
\cap[x_0,R]}> \delta R$. Then:
\begin{align*}
\frac{N(T;V,x_0)}{\sqrt{T}\log T}\to\infty
\textnormal{ as }
T\to\infty. 
\end{align*}
\end{thm}

Thus, if our goal is to find a potential $V$ such that 
$N(T;V,x_0)= \frac{1}{\pi}\sqrt{T}\log{\sqrt{T}} + \frac{1}{\pi}(2\log2 - 1 - x_0)\sqrt{T} + O(\log{T})$, we 
can summarize our findings as follows:
\begin{itemize}
 \item [(a)] Theorem \ref{T:bigoh} gives the heuristic that if $N(T;V,x_0) = O(\sqrt{T}\log \sqrt{T})$, 
 then $V(t)$ must be at least on the order of $e^{at}$ for some $a>0$. (More precisely, it can't be 
 dominated by a sub--exponential function on sets of fixed percentages of intervals $[0,R]$). 
 \item [(b)] Theorem \ref{T:mt} says that for potentials of the form $\frac{1}{4}e^{2t} + \eps(t)$ where 
 $\eps(t) > e^{(1+\eps)t}$ (or $\eps(t)< e^{(1+\eps)e}$), $N(T;V,x_0)$ will not have the desired 
 asymptotics. In particular, if $V(t)\simeq e^{at}$ and $a\neq 2$, then $N(T;V,x_0)$ does not have 
 the desired asymptotics. Thus, the heuristic is that if $V$ is a potential with the desired spectral 
 asymptotics, then $V(t)$ must be a small perturbation of $\frac{1}{4}e^{2t}$.
 \item [(c)] Finally, Theorem \ref{T:converse} shows that small perturbations of $\frac{1}{4}e^{2t}$ will not 
 produce the desired asymptotics. 
\end{itemize}

The theme of this paper is then that a potential that gives the desired Weyl asymptotics is not just a 
``small peturbation'' of a well--established potential (like the Morse potential). It seems that 
any potential that gives the desired Weyl asymptotics is going to have to oscillate wildly 
between sub and super exponential functions, will have singularities, and will probably not be a function 
(that is, it is a distribution). 



Additionally, to make much more progress in this area, a version of Theorem \ref{T:speca} 
that puts less restrictions on $V$ is needed. Theorem \ref{T:converse} is a step in this direction
and hopefully can be extended to other non--exponential potentials.

For the rest of the paper, we set $\alpha=0$ and we don't write the ``$\alpha$'' 
in $N(T;V, x_0, \alpha)$. The proofs of Theorems \ref{T:converse}, \ref{T:bigoh}, and 
\ref{T:mt} are in the following sections. Section 2 also contains some material that is 
used throughout the paper.

\section*{Acknowledgment}
Alex Poltoratski introduced me to this area and told me about several lines of investigation 
and ideas to pursue in this area. I would like to thank him for this and for discussing 
some of the thoughts herein.

\section{Proof of Theorem \ref{T:converse}}
The proof of Theorem \ref{T:converse} is based on a very simple observation along with the 
integral estimate in \cite{Lag2009}. We give some background information on which the simple
observation is based. This observation is also used in other parts of this paper. 

%

We have this basic theorem due to Sturm; see for example \cites{Titchmarsh1946, LevitanSargsjan1970}
\begin{thm}\label{T:st}
Let $u$ be a (any) solution to:
\begin{align*}
 u'' + (\lambda - g)u = 0
\end{align*}
and let $v$ be a (any) solution to:
\begin{align*}
 v'' + (\lambda - h)u = 0, 
\end{align*}
where $h(x) < g(x)$ (so that $\lambda - g(x) < \lambda - h(x)$). Between any two zeros of $u$, there is a 
zero of $v$. In particular:
\begin{align*}
 \#\{\textnormal{zeros of } v\} \geq \#\{\textnormal{zeros of } u\} - 1
\end{align*}
\end{thm}
We have the following corollary (see also, for example, \cite{Simon2005}):
\begin{cor}\label{C:evalests}
Let $g,h$ be as above. Consider the same equations as above but consider only solutions that satisfy a 
boundary condition at $0$ and are in $L^2$ (that is, we have an eigenvalue problem; we consider only the 
''limit point'' case). Then for every $a>0$ we have:
\begin{align*}
\#\{\textnormal{eigenvalues of second problem in } [0,a]\}
\geq \#\{\textnormal{eigenvalues of first problem in } [0,a]\} + O(1).
\end{align*}
\end{cor}
\begin{proof}
First, consider the two differential equations:
\begin{align}\label{E:bp}
y'' + (\lambda - g)y = 0 
\end{align}
and
\begin{align}\label{E:sp}
y'' + (\lambda - h)y = 0.
\end{align}
We want to consider solutions to these equations that satisfy $y(0)=0$ and $y'(0)=1$. Let $y(x,\lambda; g)$ 
and $y(x, \lambda; h)$ denote solutions to the respective problems. We consider the limit point case and 
we have that $y(x,\lambda; g)\in L^2$ if and only if $\lambda$ is an eigenvalue of \eqref{E:bp}; and similarly 
for $y(x,\lambda; h)$ and \eqref{E:sp}. 

For a fixed $\lambda$, we know from Theorem \ref{T:st} that 
$y(x,\lambda; h)$ has at least as many zeros as $y(x,\lambda; g)$ (up to an $O(1)$ error). Furthermore -- also 
by Theorem \ref{T:st} -- if 
$\lambda_{k}(g)$ and $\lambda_{k+1}(g)$ are the $k^{\textnormal{th}}$ and $(k+1)^{\textnormal{th}}$ eigenvalues 
for \eqref{E:bp}, then $y(x, \lambda; g)$ has $k$ zeros for $\lambda_{k}(g) < \lambda < \lambda_{k+1}(g)$; that 
is, $y(x, \lambda; g)$ gains a zero only when $\lambda$ is an eigenvalue (at which point 
it gains exactly one zero). Of course, similar statements are true for $y(x, \lambda; h)$. 

Now, to prove the Corollary \ref{C:evalests} we reason as follows. Starting with $\lambda = 0$, increase 
$\lambda$ in a continuous manner. If $\lambda$ passes through two eigenvalues of problem \eqref{E:bp} with out 
passing through an eigenvalue of \eqref{E:sp}, then we have a contradiction to Theorem \ref{T:st}. 
Indeed, if $\lambda = \lambda_2(g)$, and $\lambda_{0}(h) > \lambda_{2}(g)$ then we have that 
$y(x, \lambda; g)$ has $2$ zeros while $y(x, \lambda; h)$ has no zeros (Theorem \ref{T:st} says that 
$y(x, \lambda; h)$ should have at least one zero). Continue increasing $\lambda$ this way, noting that 
whenever it passes through an eigenvalue of problem \eqref{E:bp}, it must pass through at least one 
eigenvalue of problem \eqref{E:sp}. 
\end{proof}

We now prove Theorem \ref{T:converse}.
\begin{proof}[Proof of Theorem \ref{T:converse}]
By assumption, there is a $C>0$ such that $\frac{1}{4}e^{2t} - Ce^{t} < \eps(t) < \frac{1}{4}e^{2t} + Ce^{t}$. 
By the theorem of Lagarias $N(T; \frac{1}{4}e^{2t} \pm C e^{t}, x_0) 
= \frac{1}{\pi}\sqrt{T}\log{\sqrt{T}} + \frac{1}{\pi}(2\log2 - 1 - x_0)\sqrt{T} + O(1)$. By Corollary \ref{C:evalests}:
\begin{align*}
N(T; \frac{1}{4}e^{2t} + Ce^{t})
\leq N(T; V, x_0) 
\leq N(T; \frac{1}{4}e^{2t} - Ce^{t}).
\end{align*}
Since the outer two terms are equal, up to an $O(1)$ error, this implies that $N(T; V, x_0) 
= \frac{1}{\pi}\sqrt{T}\log{\sqrt{T}} + \frac{1}{\pi}(2\log2 - 1 - x_0)\sqrt{T} + O(1)$.
\end{proof}

\section{Proof of Theorem \ref{T:mt}}

To prove Theorem \ref{T:mt}, we use Weyl's law. This is a well--known theorem with several 
variants. We quote the one from \cite{Lag2009}:
\begin{thm}\label{T:speca}
Let $V(t) = Q_k(t)$ or $\frac{1}{4}e^{2t} + \frac{1}{4}e^{(1+\eps)t}$. Then there holds:
\begin{align}\label{E:weyl_law_2}
N(T; V, x_0)
= \frac{1}{\pi}\int_{x_0}^{V^{-1}(T)}\sqrt{T - V(t)}dt + O(1).
\end{align} 
\end{thm}

To prove Theorem \ref{T:mt} we use Weyl's law above and the following two lemmas. 

\begin{lm}\label{L:ub}
Let $V$ be a potential such that there is an $\eps>0$ such that:
\begin{align*}
V(t) > \frac{1}{4}e^{2t} + \frac{1}{4}e^{(1+\eps)t}.
\end{align*}
Then 
\begin{align*}
\abs{N(T; Q_0, x_0) - N(T; V, x_0)} \gtrsim T^{\frac{\eps}{4}}\sqrt{\log T}.
\end{align*}
\end{lm}
\begin{proof}
First, by corollary \ref{C:evalests}, $N(T; Q_0, x_0) > N(T; V, x_0)$ (since $Q_0 < V$). Also, 
since $V(t) > \frac{1}{4}e^{2t} + \frac{1}{4}e^{(1+\eps)t}$, by Corollary \ref{C:evalests}, it follows that
$N(T; \frac{1}{4}e^{2t} + \frac{1}{4}e^{t}, x_0) > N(T; V, x_0)$ and so:
\begin{align*}
\abs{N(T; Q_0, x_0) - N(T; V, x_0)}
= N(T; Q_0, x_0) - N(T; V, x_0)
> N(T; Q_0, x_0) - N(T; \frac{1}{4}e^{2t} + \frac{1}{4}e^{(1+\eps)t}, x_0),
\end{align*}
and so we get a lower bound on this last term. Letting $W(t) = \frac{1}{4}(e^{2t}+e^{(1+\eps)t})$, 
by \eqref{E:weyl_law_2}, this is equal to:
\begin{align*}
\int_{t=0}^{Q_0^{-1}(T)}\sqrt{T-Q_0(t)}dt
-\int_{t=0}^{W^{-1}(T)}\sqrt{T-W(t)}dt.
\end{align*}
For the rest of the proof, let $T=\frac{1}{4}e^{2u}$ so that $Q_{0}^{-1}(T) = u$; also let $p=W^{-1}(T)$. Thus, 
we can write the integral above as:
\begin{align}\label{E:split}
\frac{1}{2}\int_{t=0}^{p}\left(\sqrt{e^{2u}-e^{2t}} - \sqrt{e^{2u}-e^{2t}-e^{(1+\eps)t}}\right) dt
+ \frac{1}{2}\int_{t=p}^{u}\sqrt{e^{2u}-e^{2t}}dt.
\end{align}



We estimate the first integral in \eqref{E:split}. (As a side note, we mention the 
second integral is ``small'' and does not contribute much). We first estimate $u-p$. Since $p$ satisfies 
$e^{2p} + e^{(1+\eps)p} = e^{2u}$, by taking $\log$ on both sides, we find:
\begin{align*}
2p + \left(\log{(e^{2p} + e^{(1+\eps)p})}-\log{e^{2p}}\right) = 2u,
\end{align*}
so that $2(u-p) = \left(\log{(e^{2p} + e^{(1+\eps)p})}-\log{e^{2p}}\right)$. The quantity in parentheses 
is estimated as:
\begin{align*}
\log{(e^{2p} + e^{(1+\eps)p})}-\log{e^{2p}}
=\log(1+e^{(\eps-1)p})
\simeq e^{(\eps - 1)p}.
\end{align*}
Thus, $u-p \simeq e^{(\eps-1)p}$.  Let 
$p_{\delta}=p-\delta p$, where $\delta>0$ is a (small) number that will be chosen later that depends only 
on $\eps$. We will show that:
\begin{align}\label{E:toshow}
\int_{p_{\delta}}^{p}\sqrt{e^{2u} - e^{2t}} - 
\sqrt{e^{2u}-e^{2t}-e^{(1+\eps)t}}dt
\gtrsim \sqrt{p}e^{\frac{\eps}{2} p}.
\end{align}
Note that $u \simeq \log{T}$ and $e^{\frac{\eps}{2}u}\simeq T^{\frac{\eps}{4}}$. Thus, this estimate  
implies the claim since $\sqrt{p}e^{\frac{\eps}{2}p} - \sqrt{u}e^{\frac{\eps}{2}u}\to 0$ as $T\to\infty$.

The integrand is estimated as:
\begin{align*}
\sqrt{e^{2u} - e^{2t}} - \sqrt{e^{2u}-e^{2t}-e^{(1+\eps)t}}
&=\frac{2^{(1+\eps)t}}{\sqrt{e^{2u} - e^{2t}} + \sqrt{e^{2u}-e^{2t}-e^{(1+\eps)t}}}
\\& \geq \frac{e^{(1+\eps)t}}{2(\sqrt{e^{2u} - e^{2p_\delta}})}.
\end{align*}
Additionally, we have:
\begin{align*}
\int_{p_\delta}^{p}e^{(1+\eps)t}dt
= \frac{1}{1+\eps}\left(e^{(1+\eps)p} - e^{(1+\eps)p_{\delta}}\right)
\geq \frac{1}{1+\eps}e^{(1+\eps)p_\delta}(p-p_\delta)
= \frac{\delta}{1+\eps}e^{(1+\eps)(1-\delta)p}p.
\end{align*}
The ``$\geq$'' above follows from the mean value theorem. 

Furthermore using the mean value theorem again, we have
\begin{align*}
e^{2u} - e^{2p_\delta}
\leq e^{2u}(u-p_\delta)
= e^{2u}((u-p) + \delta p)
= e^{2u}(u - p(1-\delta)).
\end{align*}
Putting together these estimates, we find:
\begin{align*}
\int_{p_{\delta}}^{p}\sqrt{e^{2u} - e^{2t}} - 
\sqrt{e^{2u}-e^{2t}-e^{(1+\eps)t}}dt
&\geq \frac{\delta}{1+\eps}
    \frac{e^{(1+\eps)(1-\delta)p}p}{\sqrt{e^{2u}(u - p(1-\delta)}}
\\&\gtrsim \frac{\sqrt{\delta p}}{1+\eps}
    e^{p-u}e^{\eps p - (1+\eps)\delta p}.
\end{align*}
Now, choose $\delta$ so small that $(1+\eps)\delta < \eps/2$ so that last quantity above is 
bigger than $\frac{\sqrt{\delta p}}{1+\eps}e^{p-u}e^{\frac{\eps}{2} p}$  (since $p-u\to 0$ so $e^{p-u}\to 1$).
This completes the proof.  
\end{proof}

\begin{rem}
Note that if $\eps =0$, the estimate above is:
\begin{align*}
\frac{\sqrt{\delta p}}{1+\eps}
    e^{p-u}e^{\eps p - (1+\eps)\delta p}
=\frac{\sqrt{\delta p}}{1+\eps}
    e^{p-u}e^{-\delta p}
\to 0.
\end{align*}
\end{rem}

\begin{lm}\label{L:lb}
Let $V$ be a potential such that there is an $\eps>0$ such that:
\begin{align*}
V(t) < \frac{1}{4}e^{2t} - \frac{1}{4}e^{(1+\eps)t}.
\end{align*}
Then 
\begin{align*}
\abs{N(T; V, x_0)-N(T; Q_0, x_0)} \gtrsim T^{\frac{\eps}{4}}\sqrt{\log T}.
\end{align*}
\end{lm}
\begin{proof}
Similar reasoning above leads us to find a lower bound on:
\begin{align*}
\int_{t=u_\delta}^{u}\sqrt{e^{2u} - (e^{2t} - e^{(1+\eps)t})}-\sqrt{e^{2u} - e^{2t}}dt,
\end{align*}
where $u_\delta = u -\delta u$ and $\delta$ depends on $\eps$ and will be chosen later. We find a lower bound 
on the integrand as:
\begin{align*}
\sqrt{e^{2u} - (e^{2t} - e^{(1+\eps)t})}-\sqrt{e^{2u} - e^{2t}}
> \frac{e^{(1+\eps)t}}{\sqrt{e^{2u}-e^{2t} + e^{(1+\eps)t}}}
> \frac{e^{(1+\eps)t}}{\sqrt{e^{2u} + e^{(1+\eps)u_\delta}}}
\simeq \frac{e^{(1+\eps)t}}{e^{u}}.
\end{align*}
And so:
\begin{align*}
\int_{t=u_\delta}^{u}\sqrt{e^{2u} - (e^{2t} - e^{(1+\eps)t})}-\sqrt{e^{2u} - e^{2t}}dt
&> e^{-u}\int_{t=u_\delta}^{u}e^{(1+\eps)t}dt
\geq\frac{\delta u}{1+\eps} e^{\eps u - (1+\eps)\delta u}.
\end{align*}
As above, choose $\delta$ small enough so that $(1+\eps)\delta < \frac{1}{2}\eps$. 
\end{proof}

\section{Proof of Theorem \ref{T:bigoh}}
To prove Theorem \ref{T:bigoh} we use the following 
well--known theorem:
\begin{thm}\label{T:weyl_law}
Let $V$ be a positive potential. Then there holds:
\begin{align*}
N(T; V,x_0) \simeq \abs{\{(t,\xi): \abs{\xi}^2 + V(t) < T\}\cap\{(t,\xi):t\geq x_0\}}.
\end{align*}
\end{thm}

We also make the following observation. If $W(t)$ is sub--exponential (by which we mean 
that $\frac{\log W(t)}{t}\to 0$ as $t\to\infty$) then there is a function $\eps(t)$ with 
$\eps(t)\to 0$, $t\eps(t)\to \infty$ and $W(t) = e^{t\eps(t)}$. Indeed, we easily compute 
$\eps(t)$ by noting that $W(t) = e^{t\frac{\log W(t)}{t}}$ and so $\eps(t)=\frac{\log W(t)}{t}$.
Since $W(t)\to\infty$ we observe that $t\eps(t)=\log W(t)\to\infty$. Furthermore, 
since $W(t)$ is sub--exponential, we conclude that $\frac{\log W(t)}{t}\to 0$ as $t\to\infty$. 

The proof of Theorem \ref{T:bigoh} will follow from this observation and the following lemma.
\begin{lm}
Let $V$ and $W$ be as in Theorem \ref{T:bigoh}. Let $\eps(t)$ be a function that satisfies 
$\eps(t)\to 0$ and $t\eps(t)\to\infty$ as $t\to\infty$. Further assume that there is a 
positive number $\delta < 1$ such that for all sufficiently large $R$ there holds 
$\abs{\{t: V(t) < Ce^{t\eps(t)}\}\cap[x_0,R]}> \delta R$. 
Then:
\begin{align*}
\frac{N(T;V,x_0)}{\sqrt{T}\log T}\to\infty
\textnormal{ as }
T\to\infty. 
\end{align*} 
\end{lm}
\begin{proof}
Let $\psi(t):=t\eps(t)$ and note that by the properties of $\psi(t)$ we have that 
$\frac{\psi^{-1}(t)}{t}\to\infty$ as $t\to\infty$. 
By Theorem \ref{T:weyl_law} we have:
\begin{align*}
N(T;V,x_0) 
\simeq \abs{\{(t,\xi):\xi^2 + V(t) < T\}}
&> \abs{\{(t,\xi):\xi^2 + V(t) < \frac{T}{2}\} \cap \{(t,\xi): V(t)<Ce^{t\eps(t)}\}}
\\& > \abs{\{(t,\xi):\xi^2 + Ce^{t\eps(t)} < \frac{T}{2}\} \cap \{(t,\xi): V(t)<Ce^{t\eps(t)}\}}
\\&=\int_{t=x_0}^{\psi^{-1}\left(\log\frac{T}{2C}\right)}
  \left(T-Ce^{t\eps(t)}\right)^{\frac{1}{2}}\unit_{\{t:V(t)<Ce^{t\eps(t)}\}}(t)dt.
\end{align*}
Now, on the set over which the integral above is taken, we have that $Ce^{t\eps(t)}<\frac{T}{2}$ and 
so $T-Ce^{t\eps(t)}\simeq T$. Furthermore, the measure of the set over which the integral is being 
taken is at least $\delta \psi^{-1}\left(\log\frac{T}{2C}\right)$. Thus we have that:
\begin{align}\label{E:bo}
N(T;V,x_0) 
\gtrsim \delta \psi^{-1}\left(\log\frac{T}{2C}\right) \sqrt{T}.
\end{align}
To prove the desired claim, we need to show that $\frac{\psi^{-1}(\log\frac{T}{2C})}{\log T} \to\infty$. 
This is equivalent to showing that $\frac{\psi^{-1}(T)}{T}\to\infty$ as $T\to\infty$. But this is 
true because $\frac{\psi(T)}{T}\to 0$ as $T\to\infty$. Thus, this completes the proof. 
%
%
%
\end{proof}

\section{Remarks and Complements}
In this section, we make some concluding remarks and extend some of the results above. 

\begin{prop}
If $V(t)$ is ``exponential order'' (by which we mean there are positive constants $C,a,b$ such that 
$\frac{1}{C}e^{at}<V(t)<C e^{bt}$) then $N(T;V)\simeq \sqrt{T}\log{T}$. 
\end{prop}
\begin{proof}
First, we note that by Corollary \ref{C:evalests}, we only need to show that $N(T;e^{kt})\simeq \sqrt{T}\log{T}$ 
for all $k>0$. To do this, we use Theorem \ref{T:speca}. Thus:
\begin{align*}
N(T;e^{kt})
\simeq \int_{x_0}^{\frac{1}{k}\log T}\left(T-e^{kt}\right)^{\frac{1}{2}}dt
\gtrsim \int_{x_0}^{\frac{1}{2k}\log T} \sqrt{T} dt
\simeq \sqrt{T}\log T.
\end{align*}
It is even easier to show that $N(T;e^{kt})\lesssim \sqrt{T}\log T$. 
%
\end{proof}

Here is a proposition that says that if $V$ is super--exponential, then the Weyl Asymptotics are too small:
\begin{prop}
Let $V(t)$ be a super--exponential potential (by which we mean $\frac{\log V(t)}{t}\to\infty$ as $t\to\infty$) 
then: 
\begin{align*}
\frac{N(T;V)}{\sqrt{T}\log T}\to 0
\textnormal{ as }
T\to\infty. 
\end{align*}
\end{prop}
\begin{proof} 
Similar to above, we assume that $V(t)>e^{t\eps(t)}$ where $\eps(t)\to\infty$; let $\psi(t)=t\eps(t)$. By Corollary 
\ref{C:evalests}, we get a upper bound on $N(T; e^{t\eps(t)})$:
\begin{align*}
\int_{x_0}^{\psi^{-1}(\log T)}\left(T-e^{t\eps(t)}\right)^{\frac{1}{2}}dt
\leq \sqrt{T}\psi^{-1}(\log T).
\end{align*}
Now, $\frac{\psi^{-1}(t)}{t}\to 0$ as $t\to\infty$ since $\frac{\psi(t)}{t}\to\infty$ as $t\to\infty$. Thus, 
\begin{align*}
\frac{\sqrt{T}\psi^{-1}(\log T)}{\sqrt{T}\log T}=\frac{\psi^{-1}(\log T)}{\log T}\to 0.
\end{align*}
\end{proof}

Finally, we briefly discuss an extension to Theorem \ref{T:bigoh}. Recall that in the proof of 
Theorem \ref{T:bigoh} we had estimate \eqref{E:bo}:
\begin{align*}
N(T;V,x_0) 
\gtrsim \delta \psi^{-1}\left(\log\frac{T}{2C}\right) \sqrt{T}.
\end{align*}
Now, let's make $\delta$ be a function that depends on $R$. That is, we know that $V$ is sub--exponential 
on sets of size $\delta(R) R$ on the intervals $[0,R]$. Then the above estimate is:
\begin{align*}
N(T;V,x_0) 
\gtrsim \delta(\log\frac{T}{2C}) \psi^{-1}\left(\log\frac{T}{2C}\right) \sqrt{T}.
\end{align*}
Thus, $\delta$ can be a decreasing function, so long as:
\begin{align*}
\frac{\delta(\log\frac{T}{2C}) \psi^{-1}\left(\log\frac{T}{2C}\right)}{\log T}\to\infty.
\end{align*}
So, for example, if $\eps(t)=t^{-\eps}$, then $\psi(t)=t^{1-\eps}$ and $\psi^{-1}(t)=t^{1+\gamma}$ for 
some $\gamma>0$. Then the estimate above is:
\begin{align*}
\frac{\delta(\log\frac{T}{2C}) \psi^{-1}\left(\log\frac{T}{2C}\right)}{\log T}
=\frac{\delta(\log\frac{T}{2C}) \left(\log\frac{T}{2C}\right)^{1+\gamma}}{\log T}
=\delta(\log\frac{T}{2C}) \left(\log\frac{T}{2C}\right)^{\gamma}.
\end{align*}
So, this still goes to $\infty$ if, for example, $\delta(t)>t^{-\frac{\gamma}{2}}$.




\end{document}